%% file: main_OscillIntegralDecay.tex
\documentclass[12pt]{amsart}
\usepackage{amssymb, amsfonts,amscd,graphicx}
\usepackage[all]{xy}
\usepackage{color}
\usepackage{tikz}

\input{macros}

\author[Cluckers]{Raf Cluckers}
\address{Universit\'e Lille 1, Laboratoire Painlev\'e, CNRS - UMR 8524, Cit\'e Scientifique, 59655
Villeneuve d'Ascq C'edex, France, and, Katholieke Universiteit Leuven, Departement wiskunde,
Celestijnenlaan 200B, B-3001 Leu\-ven, Bel\-gium.
}
\email{raf.cluckers@wis.kuleuven.be}
\urladdr{http://www.wis.kuleuven.be/algebra/Raf/}

\author[Miller]{Daniel~J.~Miller}
\address{Emporia State University, Department of Mathematics, Computer Science and Economics, 1200 Commercial Street, Campus Box 4027, Emporia, KS 66801, U.S.A.}
\email{dmille10@emporia.edu}

\thanks{During the preparation of this paper, the research of the first author has been partially supported
by the Fund for Scientific Research – Flanders (G.0415.10), Belgium, and the research
of the second author has been partially supported by the National Science Foundation (award
number 1101248), USA}



\title[Decay of Oscillatory Integrals]{Bounding the decay of oscillatory integrals with a constructible amplitude function and a globally subanalytic phase function}

\begin{document}
\maketitle

\begin{abstract}
We call a function \emph{constructible} if it has a globally subanalytic domain and can be expressed as a sum of products of globally subanalytic functions and logarithms of positively-valued globally subanalytic functions.  Our main theorem gives uniform bounds on the decay of parameterized families of oscillatory integrals with a constructible amplitude function and a globally subanalytic phase function, assuming that the amplitude function is integrable and that the phase function satisfies a certain natural condition called the hyperplane condition.  As a simple application of this theorem, we also show that any continuous, integrable, constructible function of a single variable has an integrable Fourier transform.
\end{abstract}

\input{intro}

\input{constr}

\input{proofs}

\bibliographystyle{amsplain}
\bibliography{bibliotex}
\end{document}

%% file: macros.tex
\newtheoremstyle{slthm}
  {9pt}
  {5pt}
  {\slshape}
  {}
  {\bfseries}
  {.}
  {.5em}
  {\thmname{#1} \thmnumber{#2}{\rm \thmnote{ (#3)}}}

\newtheoremstyle{prcl}
  {9pt}
  {5pt}
  {\slshape}
  {}
  {\bfseries}
  {.}
  {.5em}
  {\thmname{#3} \thmnumber{ #2}}

\theoremstyle{slthm}
\newtheorem{theorem}{Theorem}[section]
\newtheorem{lemma}[theorem]{Lemma}

\newtheorem{proposition}[theorem]{Proposition}

\theoremstyle{definition}
\newtheorem{definition}[theorem]{Definition}

\theoremstyle{remark}

\newtheorem{notationConvention}[theorem]{Notation and Conventions}

\newtheorem{remark}[theorem]{Remark}

\numberwithin{equation}{section}

\allowdisplaybreaks[2]

\newcommand{\A}{\mathcal{A}}

\newcommand{\C}{\mathcal{C}}

\newcommand{\NN}{\mathbb{N}}

\newcommand{\QQ}{\mathbb{Q}}
\newcommand{\RR}{\mathbb{R}}
\newcommand{\CC}{\mathbb{C}}

\newcommand{\tld}[1]{\widetilde{#1}}

\newcommand{\Restr}[1]{\big|_{#1}}

\newcommand{\PD}[3]{\frac{\partial^{#1}#2}{\partial {#3}^{#1}}}

\newcommand{\PDn}[3]{\frac{\partial^{|#1|}#2}{\partial {#3}^{#1}}}

\DeclareMathOperator{\cl}{cl}

\DeclareMathOperator{\vol}{vol}

%% file: intro.tex
\section{Introduction}\label{s:intro}

The classical Riemann-Lebesgue lemma states that the Fourier transform, $\hat{f}$, of a Lebesgue integrable function $f$ tends to $0$ at $\infty$.  If additional assumptions are imposed upon $f$, bounds for the decay rate of $\hat{f}$ at $\infty$ can be given.  One simple reason the decay rate of $\hat{f}$ at $\infty$ is of interest is that it is closely related to the question of whether $\hat{f}$ is Lebesgue integrable.  And in turn, one reason the integrability of $\hat{f}$ is of interest is that when $\hat{f}$ is integrable, we can recover the original function $f$ almost everywhere by taking the inverse Fourier transform of $\hat{f}$, with the inverse Fourier transform being defined simply as a Lebesgue integral.

Questions about decay rates of Fourier transforms have been studied in the wider context of oscillatory integrals, by which we shall mean any function $F:\RR^n\to\CC$ of the form
\[
F(z) = \int_{\RR^m} f(y) e^{iz\cdot \phi(y)} dy,
\]
where $i=\sqrt{-1}$, $y = (y_1,\ldots,y_m)$ and $z = (z_1,\ldots,z_n)$ are tuples of variables, $dy$ refers to the Lebesgue measure on $\RR^m$, the $\cdot$ denotes the Euclidean inner-product on $\RR^n$, and the {\bf\emph{amplitude function}} $f:\RR^m\to\RR$ and {\bf\emph{phase function}} $\phi:\RR^m\to\RR^n$ are Lebesgue measurable, with $f$ being Lebesgue integrable.  An extensive theory of oscillatory integrals has been developed both over the reals (see e.g. Arnold, Gusein-Zade and Varchenko's book \cite{AGZV88} and Stein's book \cite{Stein}) and also over the $p$-adics (see e.g. Igusa's book \cite{Igusa3} and the articles
Cluckers \cite{Cluckers2004} \cite{CCorput}, Denef \cite{Denef1991}, and Lichtin \cite{Lichtin}).  Our work here is conducted over the reals, and it is most directly inspired by a theorem of Stein (see Theorem 2 in Chapter VIII of Stein \cite{Stein}, found on page 351; Stein in turn references the preprint \cite{Bjork1973} by Bj\"{o}rk), and by analoguous $p$-adic results in \cite{Cluckers2004} and \cite{CCorput}.  Stein's theorem states that if $f$ is a $C^\infty$ function with compact support and if $\phi$ is a $C^\infty$ function that is of ``finite type'' on the support of $f$ (which is a certain condition on the partial derivatives of $\phi$), then there exist positive constants $c$ and $p$ such that
\begin{equation}\label{eq:decayBndStein}
|F(z)|\leq c\|z\|^{-p}
\end{equation}
for all nonzero $z\in\RR^n$, where $\|z\|$ denotes the Euclidean norm of $z$.

Stein's proof of this theorem uses smooth partitions of unity and the compactness of the support of $f$ to reduce the problem to a local analysis of $\phi$, which then ultimately relies on the van der Corput lemma.  This method of reducing to a local analysis may seem like a very natural technique from the viewpoint of real analysis, but since $f$ is required to be both smooth and have compact support, it renders the theorem inapplicable to many functions of interest in real analytic geometry and o-minimal structures.  For example, if $f$ happens to be semialgebraic, then requiring $f$ to be $C^\infty$ would mean that $f$ is in fact analytic, and then additionally requiring $f$ to have compact support would mean that $f$ is identically zero, which is a case void of any interest.  Stein does point out that actually $f$ need only be sufficiently differentiable for the proof to go through, but this still severely limits the theorem's applicability to many functions $f$ of interest in real analytic geometry, which very often do not have compact support.  However, because this critique of the theorem is solely due to its lack of applicability in certain natural o-minimal contexts, it makes sense to try to adapt the theorem to situations where $f$ and $\phi$ are assumed to be constructed from functions definable in certain o-minimal structures, for then one has a whole new set of tools with which to work.  That is precisely what we do in this paper, and as is common in o-minimality, our method easily adapts to give a parameterized version of the theorem with little additional effort.  So in addition to the coordinates $y = (y_1,\ldots,y_m)$ and $z = (z_1,\ldots,z_n)$, let us also consider parameter variables $x = (x_1,\ldots,x_k)$ varying over a set $X\subset\RR^k$.

\subsection{The Main Results}
Stein's finite type assumption on the phase function will be replaced by the following property in our context.

\begin{definition}\label{def:hypCond}
We say that a Lebesgue measurable function $\phi:X\times\RR^m\to\RR^n$ {\bf\emph{satisfies the hyperplane condition over $X$}} if for every $x\in X$ and every hyperplane $H$ in $\RR^n$, the set $\{y\in\RR^m : \phi(x,y)\in H\}$ has Lebesgue measure zero.
\end{definition}

The general o-minimal framework is too general for our purposes here, and we focus on functions that naturally arise in subanalytic geometry. We call a function {\bf\emph{constructible}} if it has a globally subanalytic domain and can be expressed as a sum of products of globally subanalytic functions and logarithms of positively-valued globally subanalytic functions (see Definitions \ref{def:subanalytic} and \ref{def:constructible}).  The following theorem is our main result.

\begin{theorem}\label{thm:decayBnd}
Let $f:X\times\RR^m\to\RR$ be a constructible function such that $f(x,\cdot):y\mapsto f(x,y)$ is Lebesgue integrable on $\RR^m$ for all $x\in X$, and let $\phi:X\times\RR^m\to\RR^n$ be a globally subanalytic function which satisfies the hyperplane condition over $X$.  Define $F:X\times\RR^n\to\CC$ by
\begin{equation}\label{eq:oscInt}
F(x,z) = \int_{\RR^m} f(x,y) e^{i z\cdot \phi(x,y)}dy.
\end{equation}
Then there exist a constant $p > 0$ and a globally subanalytic function $g:X\to(0,\infty)$ such that for all $(x,z)\in X\times\RR^n$ with $z\neq 0$,
\begin{equation}\label{eq:decayBnd}
|F(x,z)| \leq g(x)\|z\|^{-p}.
\end{equation}
\end{theorem}

The hyperplane condition is easily seen to be a necessary assumption in Theorem \ref{thm:decayBnd} in the sense that if $\phi$ does not satisfy the hyperplane condition, then there exist a constructible function $f$ and a choice of parameter $x\in X$ such that $|F(x,\cdot)|$ is constant and nonzero on an unbounded set, in which case a bound of the form \eqref{eq:decayBnd} is impossible (see Remark \ref{rmk:decayBnd}).

Although Theorem \ref{thm:decayBnd} and Stein's theorem are both variants of the same theme, there is a certain trade-off between the two statements.  At the cost of assuming that $f$ and $\phi$ belong to very special model-theoretically defined classes of functions, Theorem \ref{thm:decayBnd} is able to significantly relax the smoothness properties of $f$ and $\phi$ assumed in Stein's theorem.  In particular, note that in Theorem \ref{thm:decayBnd},  $f$ and $\phi$ need not be continuous, let alone $C^\infty$, and their supports and ranges need not be bounded.   Even so, one should be aware that the fact that $f$ and $\phi$ are respectively assumed to be constructible and globally subanalytic automatically imposes on them a whole host of special analytic properties, such as being piecewise analytic functions and having simple asymptotic behavior at the boundaries of these pieces, and these special properties make Theorem \ref{thm:decayBnd} possible.

Another point of comparison between the two theorems concerns the value of $p$.  Stein's theorem uses the finite type assumption on the phase function to actually give a specific value of $p$ for which \eqref{eq:decayBndStein} holds, and this value of $p$ works for any choice of smooth, compactly supported amplitude function. (The choice of amplitude function only affects the value of the constant $c$.)  In contrast, Theorem \ref{thm:decayBnd} only proves the \emph{existence} of some $p>0$ for which \eqref{eq:decayBnd} holds without ever naming any particular value of $p$ that works. Nevertheless, the fact that Theorem \ref{thm:decayBnd} is a simply stated result about the classes of constructible and globally subanalytic functions as a whole can make up for its lack of explicitness regarding the value of $p$, and as a demonstration of this utility we prove the following result as an easy consequence of Theorem \ref{thm:decayBnd}.

\begin{theorem}\label{thm:integFT}
If a constructible function $f:\RR\to\RR$ is Lebesgue integrable and continuous, then the Fourier transform of $f$ is Lebesgue integrable.
\end{theorem}

Theorem \ref{thm:integFT} easily adapts to complex-valued functions as well, and the theorem is sharp in the sense that up to almost everywhere equivalence, the assumption that $f$ is continuous is a rather apparent necessary condition for the Fourier transform of $f$ to be integrable (see Remark \ref{rmk:integFT}).  An open question of interest to the authors is whether Theorem \ref{thm:integFT}, or some suitable variant of this result, is true for all multivariate constructible functions $f:\RR^m\to\RR$.

We hope that our work on oscillatory integrals in a subanalytic context may lead to a more general study of parameter oscillatory integrals in this context, in analogy to the $p$-adic and motivic study \cite{CLexp} following \cite{Cluckers2004}.

\subsection{The Method of Proof}
We conclude the Introduction by explaining how our proof of Theorem \ref{thm:decayBnd} relates to Stein's proof of his theorem.  Write $\lambda = \|z\|$.  Our proof constructs a certain subanalytic set $E\subset(1,\infty)\times X\times\RR^m$ such that for each $x\in X$, the functions $f(x,\cdot)$ and $\phi(x,\cdot)$ are well enough behaved on the fiber $E_{(\lambda,x)} := \{y\in\RR^m : (\lambda,x,y)\in E\}$ so that we may roughly follow the outline of Stein's proof of his theorem on $E_{(\lambda,x)}$, and such that the complementary set $\RR^m\setminus E_{(\lambda,x)}$ tends to a set of measure zero as $\lambda\to\infty$.  Although we have good control of $f(x,\cdot)$ and $\phi(x,\cdot)$ on $E_{(\lambda,x)}$, we also need to analyze our integrand on $\RR^m\setminus E_{(\lambda,x)}$.  To do this, we simply use the bound
\[
\left|\int_{\RR^m\setminus E_{(\lambda,x)}} f(x,y)e^{i z\cdot\phi(x,y)}dy\right| \leq \int_{\RR^m\setminus E_{(\lambda,x)}} |f(x,y)|dy
\]
and then use results from \cite{CluckersMiller01} and \cite{CluckersMiller03} to prove a bound of the form
\[
\int_{\RR^m\setminus E_{(\lambda,x)}} |f(x,y)|dy \leq g(x)\lambda^{-p},
\]
albeit with no control over the choice of $p$. This method is a mixture of ingredients of Stein's proof of Theorem 2 in Chapter VIII of \cite{Stein} with the method of \cite{Cluckers2004}, where a more stringent form of the hyperplane condition is used, namely a form of essential surjectivity (dominancy) of the phase function, for the analogous $p$-adic results. Since meanwhile the $p$-adic van der Corput Lemma has been achieved in sufficient generality in \cite{CCorput}, the results of \cite{Cluckers2004} can probably be generalized to the generality of the hyperplane condition with similar arguments as in this paper.

The choice of $p$ could, perhaps, be made more explicit in future research because the requisite results from \cite{CluckersMiller01} and \cite{CluckersMiller03} have rather explicit constructive proofs that could be analyzed more carefully, but this appears to be a nontrivial task at the present.  However, one thing is currently clear from simple examples: because we do not assume that the amplitude function is smooth, the choice of $p$ depends on both the amplitude function and the phase function, not just on phase function alone as in Stein's theorem.

%% file: constr.tex
\section{Constructible functions and their basic properties}\label{s:constr}

This section defines the notion of a constructible function and lists the properties of constructible functions that we shall need, many of which are quoted from results proven in \cite{CluckersMiller01} and \cite{CluckersMiller03}.\footnote{The reader interested in working through the supporting results in \cite{CluckersMiller01} and \cite{CluckersMiller03} may also find the simple expository paper \cite{CluckersMiller02} to be helpful.  This paper extends the main results of \cite{CluckersMiller01} through a much simpler argument than that given in \cite{CluckersMiller01}.  It also gives a simple single-variate version of the more general, but highly technical, multi-variate construction used to prove the preparation theorem of \cite{CluckersMiller03}.}
The section concludes with an asymptotic bound that is key to the proof of Theorem \ref{thm:decayBnd}.  We begin by fixing some notation and conventions to be used throughout the paper.

\begin{notationConvention}\label{nc:general}
Write $\NN = \{0,1,2,3\ldots\}$ for the set of natural numbers, and write $\log:(0,\infty)\to\RR$ for the natural logarithm.  For any set $A\subset\RR^m$, write $\cl(A)$ for the topological closure of $A$ in $\RR^m$.  For any function $f:\RR\to\CC$, write
\[
f(a^+) =
\begin{cases}
\displaystyle \lim_{x\to a^+} f(x),
    & \text{if $a\in\RR$,} \vspace*{5pt}\\
\displaystyle \lim_{x\to-\infty} f(x),
    & \text{if $a=-\infty$,}
\end{cases}
\quad
\text{and}
\quad
f(a^-) =
\begin{cases}
\displaystyle \lim_{x\to a^-} f(x),
    & \text{if $a\in\RR$,} \vspace*{5pt}\\
\displaystyle \lim_{x\to+\infty} f(x),
    & \text{if $a=+\infty$.}
\end{cases}
\]

The words {\bf\emph{measurable}} and {\bf\emph{integrable}} are always meant in the sense of the Lebesgue measure and the Lebesgue integral.  When $f:\RR\to\CC$ is integrable, define its Fourier transform $\hat{f}:\RR\to\CC$ and inverse Fourier transform $\check{f}:\RR\to\CC$ by
\[
\hat{f}(z) = \frac{1}{\sqrt{2\pi}}\int_{\RR} f(y) e^{-iyz}dy
\quad\text{and}\quad
\check{f}(z) = \frac{1}{\sqrt{2\pi}}\int_{\RR} f(y) e^{iyz}dy.
\]

Unless explicitly stated otherwise, we will write $x = (x_1,\ldots,x_k)$, $y = (y_1,\ldots,y_m)$ and $z = (z_1,\ldots,z_n)$ for the standard coordinates on $\RR^k$, $\RR^m$ and $\RR^n$, respectively, as was done in the Introduction.  Define the coordinate projection $\Pi_k:\RR^{k+m}\to\RR^k$ by $\Pi_k(x,y) = x$.  For any $A\subset\RR^{k+m}$ and $x\in\RR^k$, write
\[
A_x = \{y\in\RR^m : (x,y)\in A\}
\]
for the fiber of $A$ over $x$.  For any function $f:A\to\RR^l$ with $A\subset\RR^{k+m}$ and any $x\in\Pi_k(A)$, write $f(x,\cdot)$ for the function from $A_x$ to $\RR^l$ defined by $y\mapsto f(x,y)$.  For each $j\in\{1,\ldots,m\}$ and $\Box\in\{<,\leq,>,\geq\}$, write $y_{\Box j}$ for $(y_i)_{i\, \Box\, j}$.  For example, $y_{<j} = (y_1,\ldots,y_{j-1})$ and $y_{\leq j} = (y_1,\ldots,y_j)$.

Write $|\xi|$ for the absolute value of a real number $\xi$, respectively, the complex modulus of a complex number $\xi$.  For any $y = (y_1,\ldots,y_m)\in\RR^m$ and $\alpha = (\alpha_1,\ldots,\alpha_m)\in\RR^m$, write
\begin{eqnarray*}
|y|
    & = &
    (|y_1|,\ldots,|y_m|), \\
\|y\|
    & = &
    \sqrt{y_{1}^{2} + \cdots + y_{m}^{2}}, \\
\log y
    & = &
    (\log y_1,\ldots,\log y_m), \quad\text{provided that $y\in(0,\infty)^m$,} \\
|\alpha|
    & = &
    \alpha_1+\cdots+\alpha_m, \\
y^\alpha
    & = &
    y_{1}^{\alpha_1}\cdots y_{m}^{\alpha_m},
\end{eqnarray*}
with $y_{i}^{\alpha_i}$ defined provided that $y_i\geq 0$ if $\alpha_i$ is not an integer, and provided that $y_i\neq 0$ if $\alpha_i < 0$.  There is a conflict of notation between this use of $|y|$ and $|\alpha|$, but the context will always distinguish the meaning: if $\alpha$ is a tuple of exponents of a tuple of real numbers, then $|\alpha|$ means $\alpha_1+\cdots+\alpha_n$; if $y$ is a tuple of real numbers not used as exponents, then $|y|$ means $(|y_1|,\ldots,|y_n|)$.  These notations may be combined, such as with $|y|^\alpha = |y_1|^{\alpha_1}\cdots |y_n|^{\alpha_n}$ and $(\log|y|)^\alpha = (\log|y_1|)^{\alpha_1}\cdots(\log|y_n|)^{\alpha_n}$.

Suppose that $\{E_\lambda\}_{\lambda>a}$ is a family of subsets of a set $A\subset\RR^k$ with real parameter $\lambda>a$. Write
\[
\text{$E_\lambda\nearrow A$ as $\lambda\to\infty$}
\]
to mean that $E_{\lambda} \subset E_{\lambda'}$ whenever $\lambda<\lambda'$ and that $A = \bigcup_{\lambda>a} E_\lambda$.  For a set $B\subset A$, write
\[
\text{$E_\lambda\searrow B$ as $\lambda\to\infty$}
\]
to mean that $E_{\lambda} \supset E_{\lambda'}$ whenever $\lambda<\lambda'$ and that $B = \bigcap_{\lambda>a}E_\lambda$.  If $E_\lambda\nearrow A'$ for some $A'\subset A$ for which $A\setminus A'$ has measure zero, write
\[
\text{$E_\lambda\nearrow A$ a.e. as $\lambda\to\infty$,}
\]
where ``a.e.'' stands for ``almost everywhere''.  Likewise, if $E_\lambda\searrow B'$ for some $B'\supset B$ for which $B'\setminus B$ has measure $0$, write
\[
\text{$E_\lambda\searrow B$ a.e. as $\lambda\to\infty$.}
\]
\end{notationConvention}

\begin{definition}\label{def:analytic}
For any set $E\subset\RR^k$, call a function $f:E \to \RR^l$ {\bf\emph{analytic}} if it extends to an analytic function on a neighborhood of $E$ in $\RR^k$.  An analytic function $u:E\to\RR$ is called a {\bf\emph{unit}} if either $u(x)>0$ for all $x\in E$, or $u(x) < 0$ for all $x\in E$.
\end{definition}

\begin{definition}\label{def:subanalytic}
A {\bf\emph{restricted analytic function}} is a function $f:\RR^k\to\RR$ such that the restriction of $f$ to $[-1,1]^k$ is analytic and $f(x) = 0$ on $\RR^k\setminus[-1,1]^k$.  Call a set or function {\bf\emph{subanalytic}} if, and only if, it is definable in the expansion of the real field by all restricted analytic functions.
\end{definition}

Thus in this paper, the word ``subanalytic'' is an abbreviation for the phrase ``globally subanalytic'', and in this meaning, the natural logarithm $\log:(0,\infty)\to\RR$ is not subanalytic.

\begin{definition}\label{def:constructible}
For any subanalytic set $X\subset\RR^k$, let $\C(X)$ denote the $\RR$-algebra of functions on $X$ generated by the functions of the form $x\mapsto f(x)$ and $x\mapsto \log g(x)$, where $f:X\to\RR$ and $g:X\to(0,\infty)$ are subanalytic.
The functions in  $\C(X)$ are called the {\bf\emph{constructible functions on $X$}}.  If we simply call a function $f$ ``constructible'', we mean that $f$ has a subanalytic domain $X$ and that $f\in\C(X)$.
\end{definition}

\begin{theorem}[{\cite[Theorem 1.3]{CluckersMiller01}, Stability under integration}]\label{thm:stabInteg}
If $f\in\C(X\times\RR^m)$ for a subanalytic set $X\subset\RR^k$, and if $f(x,\cdot)$ is integrable on $\RR^m$ for all $x\in X$, then the function $F:X\to\RR$ defined by
\begin{equation}\label{eq:paramInt}
F(x) = \int_{\RR^m} f(x,y)\, dy
\end{equation}
is constructible.
\end{theorem}

It is easy to express any constructible function $F:X\to\RR$ as an integral of the form \eqref{eq:paramInt} for some subanalytic function $f:X\times\RR^m\to\RR$, so Theorem \ref{thm:stabInteg} shows that the constructible functions form the smallest class of functions that contains the subanalytic functions and is stable under integration.

\begin{definition}\label{def:cell}
A set $A\subset\RR^{k+m}$ is called a {\bf\emph{cell over $\RR^k$}} if $A$ is subanalytic and for each $j\in\{1,\ldots,m\}$, either $\Pi_{k+j}(A)$ is the graph of an analytic subanalytic function on $\Pi_{k+j-1}(A)$ or
\[
\Pi_{k+j}(A) = \{(x,y_{\leq j}) : (x,y_{<j})\in\Pi_{m+j-1}(A), a_j(x,y_{<j})\,\, \Box_1 \,\, y_j \,\, \Box_2 \,\, b_j(x,y_{<j})\}
\]
for some analytic subanalytic functions $a_j,b_j:\Pi_{k+j-1}(A)\to\RR$ such that $a_j(x,y_{<j}) < b_j(x,y_{<j})$ for all $(x,y_{<j})\in\Pi_{k+j-1}(A)$, where $\Box_1$ and $\Box_2$ either denote $<$ or no condition.  We say that $A$ is {\bf\emph{open over $\RR^k$}} if the fiber $A_x$ is open in $\RR^m$ for all $x\in\Pi_k(A)$.
\end{definition}

\begin{definition}\label{def:center}
Suppose that $A\subset\RR^{k+m}$ is a cell over $\RR^k$ that is open over $\RR^k$.  For each $j\in\{1,\ldots,m\}$, suppose that we have an analytic subanalytic function $\theta_j:\Pi_{k+j-1}(A)\to\RR$ whose graph is disjoint from $\Pi_{k+j}(A)$, and write $\tld{y}_j = y_j - \theta_j(x,y_{<j})$.  In this situation, we call $\theta = (\theta_1,\ldots,\theta_m):A\to\RR^m$ a {\bf\emph{center for $A$ over $\RR^k$}}, and we call $(x,\tld{y}) = (x,\tld{y}_1,\ldots,\tld{y}_m)$ the {\bf\emph{coordinates for $A$ with center $\theta$}}.\footnote{The notion of a ``center for $A$ over $\RR^k$'' is also found in \cite{CluckersMiller03}, but with additional technical conditions that we are omitting here because we have no need for them.}
\end{definition}

\begin{definition}\label{def:ratMonMap}
Suppose that $(x,\tld{y})$ are the coordinates for $A\subset\RR^{k+m}$ with center $\theta$.  A {\bf\emph{rational monomial map on $A$ over $\RR^k$ with center $\theta$}} is a bounded function $\varphi:A\to(0,\infty)^M$ of the form
\begin{equation}\label{eq:ratMonMap}
\varphi(x,y) = (c_1(x)|\tld{y}|^{\mu_1},\ldots,c_M(x)|\tld{y}|^{\mu_M}),
\end{equation}
where $c_1,\ldots,c_M:\Pi_k(A)\to(0,\infty)$ are analytic subanalytic functions and $\mu_1,\ldots,\mu_M\in\QQ^m$.  We call $f:A\to\RR$ a {\bf\emph{$\varphi$-function}} if $f = F\circ\varphi$ for some analytic function $F:\cl(\varphi(A))\to\RR$.  If $F$ is also a unit, then we call $f$ a {\bf\emph{$\varphi$-unit}}.
\end{definition}

\begin{theorem}[Preparation of constructible functions]
\label{thm:constrPrep}
Let $f\in\C(X\times\RR^m)$ for a subanalytic set $X\subset\RR^k$, and assume that $f(x,\cdot)$ is integrable on $\RR^m$ for all $x\in X$.  Then there exists a finite partition $\A$ of $X\times\RR^m$ into subanalytic cells over $\RR^k$ such that for each $A\in\A$ which is open over $\RR^k$, there exists a rational monomial map $\varphi$ on $A$ over $\RR^k$, say with center $\theta$, for which we can express $f$ as a finite sum
\begin{equation}\label{eq:constrPrepSum}
f(x,y) = \sum_{j}T_j(x,y)
\end{equation}
on $A$, where for each $j$, the function $T_j(x,\cdot)$ is integrable on $A_x$ for all $x\in\Pi_k(A)$, the function $T_j$ has a constant positive or negative sign on $A$, and
\begin{equation}\label{eq:constrPrepTerm}
T_j(x,y) = g_j(x)|\tld{y}|^{\alpha_j}(\log|\tld{y}|)^{\beta_j} u_j(x,y)
\end{equation}
for some $g_j\in\C(\Pi_k(A))$, tuples $\alpha_j\in\QQ^m$ and $\beta_j\in\NN^m$, and $\varphi$-unit $u_j$, where $(x,\tld{y})$ are the coordinates for $A$ with center $\theta$.

In addition, for each $j$, the fact that $T_j$ is integrable in $y$ is determined solely by the value of $\alpha_j$, being independent of the value of $\beta_j$ in the following sense: for each $x\in\Pi_m(A)$, any function of the form $y\mapsto |\tld{y}|^{\alpha_j}(\log|\tld{y}|)^{\beta'_j}$ is integrable on $A_x$, where $\alpha_j$ is as in \eqref{eq:constrPrepTerm} and $\beta'_j$ is an arbitrary tuple in $\NN^m$.
\end{theorem}

\begin{proof}
This is a special case of \cite[Theorem 1.3]{CluckersMiller03}, except for the requirement that each term $T_j$ has constant sign.  To achieve this additional property, note that we may assume that each function $g_j$ can be written as a finite product $h(x)\prod_i \log h_i(x)$ for subanalytic functions $h:\Pi_k(A)\to\RR$ and $h_i:\Pi_m(A)\to(0,\infty)$, since any constructible function on $\Pi_m(A)$ may be written as a finite sum of such products.  Now partition $A$ into smaller subanalytic cells over $\RR^k$ so as to assume that the functions $h$, $\log h_i$ and $\log|\tld{y}_1|,\ldots,\log|\tld{y}_m|$ all have constant sign on $A$.
\end{proof}

\begin{lemma}[{\cite[Lemma 7.1]{CluckersMiller01}}]
\label{lemma:subanalBound}
If $f\in\C(X)$ for a subanalytic set $X\subset\RR^k$, then there exists a subanalytic function $G:X\to(0,\infty)$ such that $|f(x)| \leq G(x)$ for all $x\in X$.
\end{lemma}

\begin{lemma}\label{lemma:subanalBndInteg}
If $f\in\C(X\times\RR^m)$ for a subanalytic set $X\subset\RR^k$, and if $f(x,\cdot)$ is integrable on $\RR^m$ for all $x\in X$, then there exists a subanalytic function $G:X\to(0,\infty)$ such that
\[
\int_{\RR^m}|f(x,y)|dy \leq G(x)
\]
for all $x\in X$.
\end{lemma}

\begin{proof}
Apply Theorem \ref{thm:constrPrep} to $f$.  It suffices to focus on one cell $A$, where \eqref{eq:constrPrepSum} holds on $A$.  By applying the triangle inequality
\[
\int_{A_x}|f(x,y)|dy \leq \sum_{j}\int_{A_x}|T_j(x,y)|dy
\]
and using the fact that each $T_j(x,\cdot)$ is integrable on $A_x$ for each $x$, it suffices to further focus on one term $T_j$.  Since $T_j$ has a constant sign, $|T_j|$ is constructible, so the lemma follows from Theorem \ref{thm:stabInteg} and Lemma \ref{lemma:subanalBound}.
\end{proof}

\begin{proposition}[{\cite[Proposition 1.5]{CluckersMiller01}}, Decay of constructible functions]
\label{thm:decayBound}
Let $f\in\C(X\times\RR)$ for a subanalytic set $X\subset\RR^k$, and suppose that for each $x\in X$, $f(x,\lambda)\to 0$ as $\lambda\to\infty$.  Then there exist $p > 0$ and a subanalytic function $g:X\to(0,\infty)$ such that
\[
|f(x,\lambda)| \leq \lambda^{-p}
\]
for all $x\in X$ and $\lambda \geq g(x)$.
\end{proposition}

In order to prove Theorem \ref{thm:decayBnd}, it is convenient to change notation by rewriting the function $F$ defined in \eqref{eq:oscInt} in polar form.  Namely, let $S = \{z\in\RR^n : \|z\| = 1\}$, and define $F:X\times S\times(0,\infty)\to\CC$ by
\begin{equation}\label{eq:Fpolar}
F(x,\xi,\lambda) = \int_{\RR^m} f(x,y) e^{i \lambda \, \xi\cdot\phi(x,y)} dy.
\end{equation}
So the goal is to prove that there exist $p > 0$ and a subanalytic function $g:X\to(0,\infty)$ such that
\[
|F(x,\xi,\lambda)| \leq g(x) \lambda^{-p}
\]
for all $(x,\xi,\lambda)\in X\times S \times (0,\infty)$.  The following lemma shows that for $F$ as defined in \eqref{eq:Fpolar}, only the behavior of $F(x,\xi,\lambda)$ for large values of $\lambda$ is relevant when proving Theorem \ref{thm:decayBnd}.

\begin{lemma}\label{lemma:decayBoundLarge}
Let $F$ be defined as in \eqref{eq:Fpolar}, where $\phi:X\times\RR^m\to\RR^n$ is subanalytic and $f\in\C(X\times\RR^m)$ is such that $f(x,\cdot)$ is integrable on $\RR^m$ for each $x\in X$. Suppose that there exist a rational number $p>0$ and subanalytic functions $g,h:X\to(0,\infty)$ such that
\[
|F(x,\xi,\lambda)| \leq g(x) \lambda^{-p}
\]
for all $(x,\xi,\lambda)\in X\times S\times (0,\infty)$ with $\lambda \geq h(x)$.  Then there exists a subanalytic function $G:X\to(0,\infty)$ such that
\[
|F(x,\xi,\lambda)| \leq G(x) \lambda^{-p}
\]
for all $(x,\xi,\lambda)\in X\times S\times (0,\infty)$.
\end{lemma}

\begin{proof}
By Lemma \ref{lemma:subanalBndInteg} there exists a subanalytic function $H:X\to(0,\infty)$ such that
\[
|F(x,\xi,\lambda)| \leq \int_{\RR^m}|f(x,y)|dy \leq H(x) < \frac{H(x) h(x)^p}{\lambda^p}
\]
for all $(x,\xi,\lambda)\in X\times S \times (0,\infty)$ with $0 < \lambda < h(x)$.  So define $G(x) = \max\{g(x), H(x) h(x)^p\}$.
\end{proof}

\begin{lemma}\label{lemma:paramIntegDecay}
Let $f\in\C(U)$ for a subanalytic set $U\subset\RR^{k+m}$, and suppose that $f(x,\cdot)$ is integrable on $U_x$ for all $x\in\Pi_k(U)$, and that $E\subset (a,\infty)\times U$ is a subanalytic set such that $E_\lambda \searrow \emptyset$ a.e. as $\lambda\to\infty$, where $E_\lambda = \{(x,y)\in U : (\lambda,x,y)\in E\}$ for each $\lambda>a$.  Then there exist $p > 0$ and a subanalytic function $g:\Pi_k(U)\to(0,\infty)$ such that
\[
\int_{E_\lambda}|f(x,y)|dy \leq \lambda^{-p}
\]
for all $x\in\Pi_k(U)$ and $\lambda \geq g(x)$.
\end{lemma}

\begin{proof}
Theorem \ref{thm:constrPrep} gives a partition $\A$ of $U$ into subanalytic cells over $\RR^k$ such that for each $A\in\A$ which is open over $\RR^k$, we can write $f$ as a finite sum $f(x,y) = \sum_{j} T_j(x,y)$ on $A$, where each function $T_j:A\to\RR$ is constructible and has constant sign, and $T_j(x,\cdot)$ is integrable on $A_x$ for each $x\in\Pi_k(A)$.  Since
\[
\int_{E_\lambda\cap A_x}|f(x,y)|dy \leq \sum_{j}\int_{E_\lambda\cap A_x}|T_j(x,y)|dy,
\]
it suffices to focus on one of the integrals of $|T_j|$ and show that it has a bound of the desired form.  The function $|T_j|$ is constructible because $T_j$ has constant sign, so Theorem \ref{thm:stabInteg} shows that
\begin{equation}\label{eq:intTjLambda}
(x,\lambda) \mapsto \int_{E_\lambda\cap A_x}|T_j(x,y)|dy
\end{equation}
is a constructible function on $\Pi_k(U)\times(a,\infty)$.  For each $x\in\Pi_k(A)$, \eqref{eq:intTjLambda} tends to $0$ as $\lambda\to\infty$ because $T_j(x,\cdot)$ is integrable and because $E_\lambda\cap A_x \searrow \emptyset$ a.e.\ as $\lambda\to\infty$.  Therefore we are done by applying Proposition \ref{thm:decayBound} to \eqref{eq:intTjLambda}.
\end{proof}

\begin{lemma}\label{lemma:derivUnit}
Suppose that $A\subset\RR^{k+m}$ is a cell over $\RR^k$ which is open over $\RR^k$, that $\varphi:A\to\RR^M$ is a rational monomial map with center $0$, that $f$ is a $\varphi$-function, and that $l\in\{1,\ldots,m\}$.  Then the function $y_l\PD{}{f}{y_l}$ is bounded on $A$.
\end{lemma}

\begin{proof}
Write $\varphi(x,y) = (c_1(x)y^{\mu_1},\ldots,c_M(x)y^{\mu_M})$ with $\mu_j = (\mu_{j,1},\ldots,\mu_{j,m})\in\QQ^m$ for each $j\in \{1,\ldots,M\}$, write $f = F\circ\varphi$ for an analytic function $F:\cl(\varphi(A))\to\RR$, and write $(X_1,\ldots,X_M)$ for coordinates on $\RR^M$.  Then
\[
y_l\PD{}{f}{y_l}(x,y) = \sum_{j=1}^{M}  \mu_{j,l} c_j(x) y^{\mu_j} \PD{}{F}{X_j}\circ\varphi(x,y)
\]
on $A$, and this function is bounded because the components of $\varphi$ and the first partial derivatives of $F$ are bounded.
\end{proof}

The asymptotic bounds given in the following proposition are key to our proof of Theorem \ref{thm:decayBnd}.

\begin{proposition}\label{prop:main}
Let $p>0$ and $f\in\C(U)$, where $U\subset\RR^{k+m}$ is subanalytic and $f(x,\cdot)$ is integrable on $U_x$ for all $x\in\Pi_k(U)$.  Then there exist a constant $c > 0$, a subanalytic function $G:\Pi_k(U)\to(0,\infty)$ and a subanalytic set $E\subset(1,\infty)\times U$ such that for each $x\in\Pi_k(U)$,
\[
\text{$E_{(\lambda,x)} \nearrow U_x$ a.e. as $\lambda\to\infty$,}
\]
and such that we can write $f$ as a finite sum
\begin{equation}\label{eq:mainSum}
f(x,y) = \sum_{l\in L} g_l(x,y_{<m}) h_l(x,y)
\end{equation}
on $U$ for some $g_l\in\C(\Pi_{k+m-1}(U))$ and $h_l\in\C(U)$ such that for each $(\lambda,x)\in\Pi_{1+k}(E)$,
\begin{equation}\label{eq:main_gBnd}
\int_{\Pi_{m-1}(E_{(\lambda,x)})} |g_l(x,y_{<m})|\, dy_{<m} \leq G(x)\lambda^p,
\end{equation}
and such that for each $(\lambda,x,y_{<m})\in\Pi_{1+k+(m-1)}(E)$,
\begin{equation}\label{eq:main_hBnd}
\left(\sup_{y_m\in E_{(\lambda,x,y_{<m})}} |h_l(x,y)|\right)
+ \int_{E_{(\lambda,x,y_{<m})}}\left|\PD{}{h_l}{y_m}(x,y)\right| dy_m
\leq
c\lambda^p.
\end{equation}
\end{proposition}

\begin{proof}
Apply Theorem \ref{thm:constrPrep} to $f$.  This constructs a partition $\A$ of $U$ into cells over $\RR^k$ such that for each $A\in\A$ which is open over $\RR^k$, we can write $f$ as a finite sum
\[
f(x,y) = \sum_{l\in L} T_l(x,y)
\]
on $A$ with each $T_l$ of the form
\[
T_l(x,y) = a(x)|\tld{y}|^\alpha (\log|\tld{y}|)^\beta u(x,y),
\]
where $a\in\C(\Pi_k(A))$, $\alpha = (\alpha_1,\ldots,\alpha_m)\in\QQ^m$, $\beta = (\beta_1,\ldots,\beta_m)\in\NN^m$, $u$ is a positively-valued $\varphi$-unit, and  $(x,\tld{y})$ are coordinates on $A$ with center $\theta$, with $\theta$ being the center of $\varphi$.   Fix $A\in\A$.  It suffices to prove the proposition for the restriction of $f$ to $A$.

If $A$ is not open over $\RR^k$, then we are done by taking the sum in \eqref{eq:mainSum} to consist of the single term $f(x,y)$ itself, and by taking $E$ to be the empty set.  So we may assume that $A$ is open over $\RR^k$.  If for each $l\in L$ we can construct $c_l$, $G_l$ and $E_l$ so that the conclusion of the proposition holds with the objects $T_l$, $c_l$, $G_l$ and $E_l$ in place of $f\Restr{A}$, $c$, $G$ and $E$, respectively, then we are done by putting $c = \sum_{l\in L} c_l$, $G(x) = \sum_{l\in L}G_l(x)$ and $E=\bigcap_{l\in L}E_l$.  So we may simply assume that
\[
f(x,y) = a(x)|\tld{y}|^\alpha (\log|\tld{y}|)^\beta u(x,y) \mbox{ on }A.
\]
After pulling back by the inverse of the coordinate transformation $(x,y) \mapsto (x,\sigma_1\tld{y},\ldots,\sigma_m\tld{y}_m)$ for a suitable choice of $\sigma_1,\ldots,\sigma_m\in\{-1,1\}$, we may further assume that $\theta = 0$ and that $A\subset\RR^k\times(0,\infty)^m$.  So we may write
\[
f(x,y) = g(x,y_{<m}) h(x,y),
\]
where
\[
g(x,y_{<m}) = a(x)y_{<m}^{\alpha_{<m}}(\log y_{<m})^{\beta_{<m}}
\]
and
\[
h(x,y) = y_{m}^{\alpha_m}(\log y_m)^{\beta_m} u(x,y).
\]

By further partitioning $A$, we may assume that $\{y_m : (x,y)\in A\}$ is contained in either $(0,e^{-1})$, $[e^{-1},e]$ or $(e,\infty)$, where $e$ is the base of the natural logarithm.  If $\{y_m : (x,y)\in A\} \subset [e^{-1},e]$, then we may assume that $\beta_m = 0$; this is because in this case we may write
\[
(\log y_m)^{\beta_m}u(x,y) = ((\log y_m)^{\beta_m}u(x,y) - c) + c
\]
for a sufficiently large constant $c$ so that $(\log y_m)^{\beta_m}u(x,y) - c$ and $c$ are both $\varphi$-units, so we may separate $f$ into two terms, both of which are still integrable in $y$ since the value of $\alpha$ has not changed.  In this way, we may assume that $|\log y_m|^{\beta_m} \geq 1$ on $A$ regardless of whether $\{y_m : (x,y)\in A\}$ is contained in $(0,e^{-1})$, $[e^{-1},e]$ or $(e,\infty)$.

By Lemma \ref{lemma:derivUnit} we may fix a constant $K > 1$ such that
\[
K^{-1} < u(x,y) < K
\quad\text{and}\quad
\left|y_m\PD{}{u}{y_m}(x,y)\right| < K
\]
on $U$.  Define $\psi:(0,\infty)\to(0,\infty)$ by
\[
\psi(t)
=
\begin{cases}
t^{\alpha_m},
    & \text{if $\alpha_m\neq  -1$,}
    \\
t^{-1/2},
    & \text{if $\alpha_m=-1$ and $t\leq 1$,}
    \\
t^{-2},
    & \text{if $\alpha_m=-1$ and $t > 1$.}
\end{cases}
\]
Note that
\[
K^{-1}\psi(y_m) \leq |h(x,y)|
\]
on $A$, and that
\[
(x,y_{<m})\mapsto K^{-1}\int_{A_{(x,y_{<m})}} \psi(t) dt
\]
is a positively-valued subanalytic function on $\Pi_{k+m-1}(A)$.

Consider a positive rational number $r$, to be specified later, and define
\[
E
=
\left\{(\lambda,x,y) \in (1,\infty)\times A : \lambda^{-r} < y_m < \lambda^r,
 \lambda^{-r} < K^{-1}\int_{A_{(x,y_{<m})}} \psi(t) dt\right\}.
\]
Note that regardless of the choice of $r$, the set $E$ will be subanalytic and $E_{(\lambda,x)}\nearrow A_x$ as $\lambda\to\infty$, for each $x\in\Pi_k(A)$.

For each $(\lambda,x)\in\Pi_{1+k}(E)$,
\begin{eqnarray*}
\lambda^{-r} \int_{\Pi_{m-1}(E_{(\lambda,x)})}|g(x,y_{<m})|\, dy_{<m}
    & \leq &
    \int_{\Pi_{m-1}(E_{(\lambda,x)})}|g(x,y_{<m})|
    \left(K^{-1}\int_{A_{(x,y_{<m})}} \psi(y_m)dy_m\right)dy_{<m}
    \\
    & \leq &
    \int_{\Pi_{m-1}(A_x)} |g(x,y_{<m})| \left(\int_{A_{(x,y_{<m})}} |h(x,y)|dy_m\right) dy_{<m}
    \\
    & = &
    \int_{A_x}|f(x,y)|\,dy.
\end{eqnarray*}
By Lemma \ref{lemma:subanalBndInteg} we may fix a subanalytic function $G:\Pi_k(A)\to(0,\infty)$ such that
\[
\int_{A_x}|f(x,y)|\,dy \leq G(x)
\]
for all $x\in\Pi_k(A)$.  Thus
\[
\int_{\Pi_{m-1}(E_{(\lambda,x)})}|g(x,y_{<m})|\, dy_{<m}
\leq
G(x) \lambda^r
\]
for all $(\lambda,x)\in\Pi_{k+1}(E)$.  So as long as $r$ is chosen with $r \leq p$, the bound \eqref{eq:main_gBnd} will be achieved.  To finish, we show how to also achieve \eqref{eq:main_hBnd} by proving that we can obtain bounds of the form $c\lambda^p$ for
\begin{equation}\label{eq:hBnd}
\sup_{y_m\in E_{(\lambda,x,y_{<m})}} |h(x,y)|
\end{equation}
and
\begin{equation}\label{eq:hDerivBnd}
\int_{E_{(\lambda,x,y_{<m})}} \left|\PD{}{h}{y_m}(x,y)\right| dy_m
\end{equation}
by choosing $r$ to be sufficiently small.  The proof divides into three cases. \\

\noindent\emph{Case 1.} $\{y_m : (x,y) \in A\} \subset [e^{-1},e]$.

In this case $h(x,y) = y^{\alpha_m} u(x,y)$, so
\[
|h(x,y)| \leq K y_{m}^{\alpha_m}
\]
and
\begin{eqnarray*}
\left|\PD{}{h}{y_m}(x,y)\right|
    & = &
    \left|\alpha_m y_{m}^{\alpha_m-1} u(x,y) + y^{\alpha_m} \PD{}{u}{y_m}(x,y)\right|
    \\
    & = &
    y_{m}^{\alpha_m-1}\left|\alpha_m u(x,y) + y_m \PD{}{u}{y_m}(x,y)\right|
    \\
    & \leq &
    K(|\alpha_m|+1)y_{m}^{\alpha_m-1}
\end{eqnarray*}
on $A$, which implies that \eqref{eq:hBnd} and \eqref{eq:hDerivBnd} are bounded above by constants.
\\

\noindent\emph{Case 2.} $\{y_m : (x,y) \in A\} \subset (e,\infty)$.

In this case $h(x,y) = y_{m}^{\alpha_m}(\log y_m)^{\beta_m} u(x,y)$ and $\{y_m : (\lambda,x,y)\in E\} \subset (e,\lambda^r)$.  So for all $(\lambda,x,y)\in E$,
\begin{equation}\label{eq:hBnd2}
|h(x,y)| \leq K \lambda^{r|\alpha_m|}(r\log \lambda)^{\beta_m}
\end{equation}
and
\begin{eqnarray*}
\left|\PD{}{h}{y_m}(x,y)\right|
    & = &
    \left|
    \alpha_m y_{m}^{\alpha_m-1} (\log y_m)^{\beta_m} u(x,y)
    +
    \beta_m y_{m}^{\alpha_m-1} (\log y_m)^{\beta_m-1} u(x,y)
    \phantom{\PD{}{u}{y_m}}
    \right.
    \\
    & &
    \left.
    +\,
    y_{m}^{\alpha_m}(\log y_m)^{\beta_m}\PD{}{u}{y_m}(x,y)
    \right|
    \\
    & = &
    y_{m}^{\alpha_m-1} |\log y_m|^{\beta_m-1}
    \left|
    \alpha_m(\log y_m)u(x,y) + \beta_m u(x,y) + y_m(\log y_m)\PD{}{u}{y_m}(x,y)
    \right|
    \\
    & \leq &
    \lambda^{r|\alpha_m-1|}(\log \lambda^r)^{|\beta_m-1|}(K|\alpha_m|\log\lambda^r + K\beta_m + K\log\lambda^r)
    \\
    & = &
    K \lambda^{r|\alpha_m-1|}(r\log\lambda)^{|\beta_m-1|}(r(|\alpha_m|+1)\log\lambda +\beta_m).
\end{eqnarray*}
And, since $(e,\lambda^r)$ has length less than $\lambda^r$,
\begin{equation}\label{eq:hDerivBnd2}
\int_{E_{(\lambda,x,y_{<m})}}\left|\PD{}{h}{y_m}(x,y)\right|dy_m
\leq
K \lambda^{r(|\alpha_m-1|+1)}(r\log\lambda)^{|\beta_m-1|}(r(|\alpha_m|+1)\log\lambda +\beta_m).
\end{equation}
By taking $r$ to be sufficiently small, \eqref{eq:hBnd2} and \eqref{eq:hDerivBnd2} imply that there exists a sufficiently large constant $c> 0$ such that $c\lambda^p$ bounds \eqref{eq:hBnd} and \eqref{eq:hDerivBnd} for all $(\lambda,x,y_{<m})\in \Pi_{k+m}(E)$.
\\

\noindent\emph{Case 3.} $\{y_m : (x,y) \in A\} \subset (0,e^{-1})$.

This case is very similar to Case 2, so we omit the details.
\end{proof}

%% file: proofs.tex
\section{Proofs of the main results}\label{s:proofs}

This section proves Theorems \ref{thm:decayBnd} and \ref{thm:integFT}.  We begin with two lemmas that come from Chapter VIII of Stein \cite{Stein}. The first one is directly related to the real van der Corput Lemma.

\begin{lemma}\label{lemma:vanderCorput}
Fix $\epsilon > 0$ and a positive integer $d$.  Suppose that $\phi:[a,b]\to\RR$ is a $d$-times differentiable function such that $|\phi^{(d)}(t)| \geq \epsilon$ for all $t\in[a,b]$.  If $d=1$, additionally assume that $\phi$ is twice differentiable and that $\phi'$ is monotonic.  Then for any differentiable function $f:[a,b]\to\RR$,
\[
\left|\int_{a}^{b}f(t) e^{i\lambda\phi(t)} dt\right| \leq
\frac{c_d}{(\lambda\epsilon)^{1/d}}
\left( \min\{|f(a)|, |f(b)|\} + \int_{a}^{b}|f'(t)|dt\right)
\]
for all $\lambda > 0$, where $c_d = 5\cdot2^{d-1}-2$.
\end{lemma}

\begin{proof}
Stein gives the van der Corput Lemma in \cite[Proposition 2 of Chapter VIII]{Stein} and then gives a corollary of this proposition directly thereafter on page 334.  The statement we are proving follows by applying this corollary to the amplitude function $f$ and phase function $\epsilon^{-1}\phi$, with the understanding that here we have simply noted the smoothness conditions on $\phi$ required for the proof of \cite[Proposition 2 of Chapter VIII]{Stein} to go through, and we have also used the inherent symmetry of the role of the endpoints $a$ and $b$ in this corollary.
\end{proof}

\begin{lemma}\label{lemma:homogBasis}
For any positive integer $d$, there exist unit vectors $v_1,\ldots,v_\ell$ in $\RR^m$ such that $\{(v_1\cdot y)^d,\ldots
(v_\ell\cdot y)^d\}$ is a basis for the the real vector space of all homogeneous polynomials in $y$ of degree $d$.
\end{lemma}

\begin{proof}
See \cite[Subsection 2.2.1 of Chapter VIII on page 343]{Stein}.
\end{proof}

\begin{proof}[Proof of Theorem \ref{thm:decayBnd}]
Define $F:X\times S\times(0,\infty)\to\CC$ as in \eqref{eq:Fpolar}, where $f\in\C(X\times\RR^m)$ is such that $f(x,\cdot)$ is integrable on $\RR^m$ for each $x\in X$, and $\phi:X\times\RR^m\to\RR$ is a subanalytic function satisfying the hyperplane condition over $X$.  It suffices to show that $F$ satisfies the hypothesis of Lemma \ref{lemma:decayBoundLarge}.  Begin by fixing a subanalytic set $U\subset X\times\RR^m$ such that for each $x\in X$, the fiber $U_x$ is open in $\RR^m$, $\RR^m\setminus U_x$ has measure zero, and $\phi(x,\cdot)$ restricts to an analytic function on $U_x$.  For each $(x_0,y_0)\in U$ and $\xi\in S$, the hyperplane condition implies that the analytic function $y\mapsto \xi\cdot(\phi(x_0,y) - \phi(x_0,y_0))$ on $U_{x_0}$ does not vanish identically in a neighborhood of $y_0$, so there exists a nonzero $\alpha\in\NN^m$ such that $\xi\cdot \PDn{\alpha}{\phi}{y}(x_0,y_0)\neq 0$.  Because $(x,y,\xi)\mapsto \xi\cdot\phi(x,y)$ is a subanalytic function on $U\times S$ and the subanalytic sets form a polynomially bounded o-minimal structure, Chris Miller's main theorem from \cite{CM95} implies that there exists a positive integer $N$ such that for each $(x,y,\xi)\in U\times S$ there exists a nonzero $\alpha\in\NN^m$ such that $|\alpha|\leq N$ and $\xi\cdot\PDn{\alpha}{\phi}{y}(x,y)\neq 0$.

We exploit the existence of this $N$ with a reasoning that is similar up to equation (\ref{eq:direcDeriv}) to an argument in the proof of Proposition 5 of Chapter VIII of \cite{Stein} (which is a key part of the proof of Theorem 2 of Chapter VIII of \cite{Stein}).  For each $d\in\{1,\ldots,N\}$, Lemma \ref{lemma:homogBasis} shows that we may fix unit vectors $v_{d,1},\ldots,v_{d,\ell(d)}$ in $\RR^m$ such that $\{(v_{d,1}\cdot y)^d, \ldots, (v_{d,\ell(d)}\cdot y)^d\}$ forms a basis for the real vector space of all homogeneous polynomials of degree $d$ in $y$.  Therefore for each nonzero $\alpha\in\NN^m$ with $|\alpha| \leq N$, we may write
\[
y^\alpha = \sum_{j=1}^{\ell(|\alpha|)} c_{\alpha,j} (v_{|\alpha|,j}\cdot y)^{|\alpha|}
\]
for unique $c_{\alpha,j}\in\RR$.  It follows that
\[
\PDn{\alpha}{\phi}{y}(x,y) =  \sum_{j=1}^{\ell(|\alpha|)} c_{\alpha,j} (v_{|\alpha|,j}\cdot \nabla)^{|\alpha|}(\phi)(x,y)
\]
on $U$, where $\nabla$ is the gradient operator $(\PD{}{}{y_1},\ldots,\PD{}{}{y_m})$.  Writing
\[
\Gamma = \{(d,j) : \text{$d\in\{1,\ldots,N\}$ and  $j \in\{1,\ldots,\ell(d)\}$}\},
\]
we see that for each $(x,y,\xi)\in U\times S$ there exists $(d,j)\in\Gamma$ such that
\begin{equation}\label{eq:direcDeriv}
\xi\cdot(v_{d,j}\cdot\nabla)^{d}(\phi)(x,y) \neq 0.
\end{equation}
Therefore we can define a positively-valued subanalytic function $M$ on $U$ by
\[
M(x,y) = \min_{\xi\in S} \left(\max_{(d,j)\in\Gamma} |\xi\cdot(v_{d,j}\cdot\nabla)^{d}(\phi)(x,y)|\right).
\]
We will use the function $M$ later in the proof.

For each $(d,j)\in\Gamma$, choose an orthonormal basis $\{v_{(d,j)}^{[1]},\ldots,v_{(d,j)}^{[m]}\}$ of $\RR^m$ with $v_{(d,j)}^{[m]} = v_{(d,j)}$.  Define $\tld{T}_{(d,j)}:\RR^m\to\RR^m$ and $T_{(d,j)}:\RR^{k+m}\to \RR^{k+m}$ by
\begin{eqnarray*}
\tld{T}_{(d,j)}(y)
    & = &
    v_{(d,j)}^{[1]}y_1 + \cdots + v_{(d,j)}^{[m]}y_m,
    \\
T_{(d,j)}(x,y)
    & = &
    (x,\tld{T}_{(d,j)}(y)),
\end{eqnarray*}
and put $U^{[d,j]} = T_{(d,j)}^{-1}(U)$.

For each $(d,j)\in\Gamma$, apply Proposition \ref{prop:main} to $f\circ T_{(d,j)}$ on $U^{[d,j]}$ with $p = \frac{1}{4N}$.  This constructs a constant $c^{[d,j]} > 0$, a subanalytic function $G^{[d,j]}:X\to(0,\infty)$ and a subanalytic set $E^{[d,j]} \subset(1,\infty)\times U^{[d,j]}$ such that for each $x\in X$,
\[
\text{$E^{[d,j]}_{(\lambda,x)} \nearrow U_{x}^{[d,j]}$ a.e. as $\lambda\to\infty$,}
\]
and such that we can write
\[
f\circ T^{[d,j]}(x,y) = \sum_{l\in L} g_l(x,y_{<m}) h_l(x,y)
\]
on $U^{[d,j]}$ for some $g_l\in\C(\Pi_{k+m-1}(U^{[d,j]}))$ and $h_l\in\C(U^{[d,j]})$ such that for each $(\lambda,x)\in \Pi_{1+k}(E^{[d,j]})$,
\begin{equation}\label{eq:dj_gBnd}
\int_{\Pi_{m-1}(E^{[d,j]}_{(\lambda,x)})} |g_l(x,y_{<m})|\,dy_{<m} \leq G^{[d,j]}(x) \lambda^{1/(4N)},
\end{equation}
and such that for each $(\lambda,x,y_{<m})\in\Pi_{k+m}(E^{[d,j]})$,
\begin{equation}\label{eq:dj_hBnd}
\left(\sup_{y_m \in E^{[d,j]}_{(\lambda,x,y_{<m})}} |h_l(x,y)|\right)
+
\int_{E^{[d,j]}_{(\lambda,x,y_{<m})}}\left|\PD{}{h_l}{y_m}(x,y)\right| dy_m
\leq
c^{[d,j]}\lambda^{1/(4N)}.
\end{equation}

Define
\[
E = \bigcap_{(d,j)\in\Gamma} \left\{(\lambda,x,y) \in(1,\infty)\times U : \text{$(\lambda,T_{(d,j)}^{-1}(x,y)) \in E^{[d,j]}$ and $M(x,y) > \lambda^{-1/4}$}\right\}.
\]
Note that $E$ is subanalytic and that for each $x\in X$, $E_{(\lambda,x)}\nearrow U_x$ a.e. as $\lambda\to\infty$.  Therefore Lemma \ref{lemma:paramIntegDecay} shows that there exist a constant $p>0$ and a subanalytic function $g:X\to(0,\infty)$ such that
\[
\left|
\int_{U_x \setminus E_{(\lambda,x)}} f(x,y) e^{i\lambda \xi\cdot\phi(x,y)} dy
\right|
\leq
\int_{U_x \setminus E_{(\lambda,x)}} |f(x,y)|\,dy
\leq
\lambda^{-p}
\]
for all $x\in X$, $\lambda > g(x)$ and $\xi\in S$.

To finish it suffices to find a suitable bound for
\[
\left|\int_{E_{(\lambda,x)}} f(x,y) e^{i\lambda \xi\cdot\phi(x,y)} dy
\right|
\]
for all $(x,\xi,\lambda)\in X\times S\times (1,\infty)$.  We do this by an argument that combines the above information with a method used to prove Proposition 5 of Chapter VIII of \cite{Stein}.  Fix $(x,\xi,\lambda)\in X\times S\times (1,\infty)$.  It follows from the definitions of $M$ and $E$ that we can fix a finite partition $\A$ of the fiber $E_{(\lambda,x)}$ into subanalytic sets such that for each $A\in\A$ there exists $(d,j)\in\Gamma$ such that $\tld{A} := \tld{T}_{(d,j)}^{-1}(A)$ is a cell over $\RR^0$ contained in $E^{[d,j]}_{(\lambda,x)}$, and
\[
\left|\xi\cdot(v_{d,j} \cdot \nabla)^d(\phi)(x,y)\right| > \lambda^{-1/4}
\]
for all $y\in A$.  Thus for $\tld{\phi} := \phi\circ T_{(d,j)}$,
\[
\left|\xi\cdot\PD{d}{\tld{\phi}}{y_m}(x,y)\right| > \lambda^{-1/4}
\]
for all $y\in\tld{A}$.  Write
\[
\int_{E_{(\lambda,x)}} f(x,y) e^{i\lambda \xi\cdot\phi(x,y)} dy
=
\sum_{A\in\A} \int_A f(x,y) e^{i\lambda \xi\cdot\phi(x,y)} dy.
\]
Fix $A\in\A$, choose $(d,j)\in\Gamma$ as described above for the set $A$, and write $\tld{f} = f\circ T_{(d,j)}$ and also $\tld{\phi} = \phi\circ T_{(d,j)}$ and $\tld{A} = \tld{T}^{-1}_{(d,j)}(A)$, as above.  Then
\begin{eqnarray*}
\left|\int_A f(x,y) e^{i\lambda \xi\cdot\phi(x,y)} dy\right|
    & = &
    \left|\int_{\tld{A}} \tld{f}(x,y) e^{i\lambda \xi\cdot\tld{\phi}(x,y)} dy \right|
    \\
    & \leq &
    \sum_{l\in L} \int_{\Pi_{m-1}(\tld{A})}|g_l(x,y_{<m})|\left|\int_{\tld{A}_{y_{<m}}} h_l(x,y) e^{i\lambda \xi\cdot\tld{\phi}(x,y)} dy_m\right| dy_{<m}.
\end{eqnarray*}
Also fix $l\in L$.  By the van der Corput style Lemma \ref{lemma:vanderCorput},
\[
\left|\int_{\tld{A}_{y_{<m}}} h_l(x,y) e^{i\lambda \xi\cdot\tld{\phi}(x,y)} dy_m\right|
\leq
\frac{c_N}{(\lambda\cdot\lambda^{-1/4})^{1/N}}
    \left(
    \sup_{y_m\in B_{y_{<m}}}|h_l(x,y)| + \int_{B_{y_{<m}}}\left|\PD{}{h_l}{y_m}(x,y)\right|dy_m,
    \right)
\]
for some constant $c_N$ only depending on $N$.
So, by \eqref{eq:dj_gBnd} and \eqref{eq:dj_hBnd},
\[
\left|\int_A f(x,y) e^{i\lambda \xi\cdot\phi(x,y)} dy\right|
    \leq
    |L|\frac{c_N G^{[d,j]}(x) \lambda^{1/(4N)} c^{[d,j]}\lambda^{1/(4N)}}{\lambda^{3/(4N)}}
    =
    \frac{|L|c_N c^{[d,j]}  G^{[d,j]}   (x)}{\lambda^{1/(4N)}},
\]
where $|L|$ denotes the cardinality of the finite index set $L$.  The theorem follows.
\end{proof}

Theorem \ref{thm:integFT} will be proven as a simple consequence of Theorem \ref{thm:decayBnd}, some standard facts about o-minimal structures, and the following elementary lemma about the Fundamental Theorem of Calculus and integration by parts.

\begin{lemma}\label{lemma:FTC&Parts}
Let $-\infty\leq a < b\leq+\infty$.  Suppose that $f:(a,b)\to\RR$ is differentiable, that the one-sided limits $f(a^+)$ and $f(b^-)$ exist in $\RR$, and that $f'$ has constant sign (i.e., either $-$, $0$, or $+$) on $(a,b)$.
\begin{enumerate}
\item
Then $f'$ is Lebesgue integrable on $(a,b)$, and
\[
\int_{a}^{b}f'(t)dt = f(b^-) - f(a^+).
\]

\item
Additionally suppose that $f$ is Lebesgue integrable on $(a,b)$ and that $g:(a,b)\to\CC$ is a differentiable function such that $g$ and $g'$ are bounded and such that the one-sided limits $g(a^+)$ and $g(b^-)$ exist in $\RR$.  Then $fg'$ and $f'g$ are Lebesgue integrable on $(a,b)$, and
\[
\int_{a}^{b}f(t)g'(t)dt = \left(f(b^-)g(b^-) - f(a^+)g(a^+)\right) - \int_{a}^{b}f'(t)g(t)dt.
\]
\end{enumerate}
\end{lemma}

The assumption in the lemma that $f'$ has constant sign can be significantly weakened, but we make this assumption because it is very simple and is sufficient for our needs.

\begin{proof}[Proof of Lemma \ref{lemma:FTC&Parts}]
Fix an increasing sequence of intervals $[a_j,b_j]$ converging to $[a,b]$.  For the first statement, since $f'$ has constant sign, Lebesgue's monotone convergence theorem implies that
\[
\int_{a}^{b} f'(t)dt = \lim_{j\to\infty} \int_{a}^{b}f'(t)\chi_{[a_j,b_j]}(t)dt = \lim_{j\to\infty}(f(b_j)-f(a_j)) = f(b^-)-f(a^+).
\]
For the second statement, the assumptions on $f$ and $g$ imply that $f'g$ and $fg'$ are Lebesgue integrable, so Lebesgue's dominated convergence theorem implies that
\[
\int_{a}^{b}(f(t)g'(t)+f'(t)g(t))dt
=
\lim_{j\to\infty}\int_{a}^{b}(fg)'(t) \chi_{[a_j,b_j]}(t) dt
=
f(b^-)g(b^-) - f(a^+)g(a^+).
\]
\end{proof}

\begin{proof}[Proof of Theorem \ref{thm:integFT}]
Assume that $f\in\C(\RR)$ is continuous and integrable, and let $y$ and $z$ denote single variables.  Because the Fourier transform $\hat{f}$ is continuous, to show that $\hat{f}$ is integrable on $\RR$, it suffices to bound $|\hat{f}|$ by a function that is integrable on the complement of some compact set.

Since $f$ is constructible, $f$ is definable in the expansion of the real field by all restricted analytic functions and the exponential function, which is o-minimal (see Van den Dries, Macintyre and Marker \cite{vdDMM}, or Lion and Rolin \cite{LR97}).  By o-minimality we may fix finitely many points $-\infty = a_0 < a_1 < \cdots < a_n < a_{n+1} = +\infty$ such that $f'$ is defined and has constant sign on  $(a_j,a_{j+1})$ for each $j\in\{0,\ldots,n\}$.  Because $f$ is piecewise monotonic and integrable, $f(y)\to 0$ as $|y|\to \infty$.  Since the function $y\mapsto e^{-iyz}$ and  its derivative $y\mapsto -ize^{-iyz}$ are bounded on $\RR$ for each fixed value of $z\in\RR$, Statement 2 of Lemma \ref{lemma:FTC&Parts} gives
\begin{eqnarray*}
\sqrt{2\pi} \, iz \hat{f}(z)
    & = &
    \int_{\RR} f(y) (iz)e^{-i yz}dy
    \\
    & = &
    \sum_{j=0}^{n} \int_{a_j}^{a_{j+1}} f(y)(iz)e^{-iyz} dy
    \\
    & = &
    -\sum_{j=0}^{n} \left(\left[f(y)e^{-iyz}\right]_{y=a_j}^{y=a_{j+1}}
    - \int_{a_j}^{a_{j+1}} f'(y) e^{-iyz}dy\right)
    \\
    & = &
    \int_{\RR}f'(y)e^{-iyz}dy,
\end{eqnarray*}
with the last equality following from the continuity of $f$ and the fact that $f(y)\to 0$ as $|y|\to\infty$.  Statement 1 of Lemma \ref{lemma:FTC&Parts} shows that $f'$ is integrable on $(a_j,a_{j+1})$ for each $j$, and is therefore integrable on $\RR$.  So Theorem \ref{thm:decayBnd} shows that there exist positive constants $c$ and $p$ such that
\[
\left|\int_{\RR}f'(y)e^{-iyz}dy\right|
\leq
c|z|^{-p}
\]
for all $z\neq 0$.  Thus
\[
\sqrt{2\pi}|\hat{f}(z)| \leq c|z|^{-(1+p)}
\]
for all $z\neq 0$, which implies that $\hat{f}$ is integrable on $\RR$.
\end{proof}

We now conclude the paper with two remarks that recast Theorems \ref{thm:decayBnd} and \ref{thm:integFT} in forms that explain in what sense the hypotheses of these theorems are, in fact, necessary. The arguments proving the remarks are well-known, but we give them in detail for the convenience of the reader and by way of giving context.

\begin{remark}\label{rmk:decayBnd}
A subanalytic function $\phi:X\times\RR^m\to\RR^n$ satisfies the hyperplane condition over $X$ if and only if the following statement holds.
\begin{equation}\label{eq:decayBndIFF}
\left\{\hspace*{5pt}
\text{\parbox{5.5in}{
For every $f\in\C(X\times\RR^m)$ such that $f(x,\cdot)$ is integrable on $\RR^m$ for all $x\in X$, there exist a constant $p >0$ and a subanalytic function $g:X\to(0,\infty)$ such that
\[
|F(x,z)| \leq g(x)\|z\|^{-p}
\]
for all $(x,z)\in X\times(\RR^m\setminus\{0\})$, where $F$ is defined from $f$ and $\phi$ as in \eqref{eq:oscInt}.
}}
\right.
\end{equation}
\end{remark}

\begin{proof}[Proof of Remark \ref{rmk:decayBnd}]
The fact that the hyperplane condition implies \eqref{eq:decayBndIFF} is Theorem \ref{thm:decayBnd}.  To prove the converse, suppose that $\phi$ does not satisfy the hyperplane condition over $X$.  Fix $x\in X$ and a hyperplane $H = \{z\in\RR^n : a\cdot z = b\}$, with $a\in\RR^n\setminus\{0\}$ and $b\in\RR$, such that
\begin{equation}\label{eq:H+meas}
\{y\in\RR^m : \phi(x,y)\in H\}
\end{equation}
has positive measure.  The set \eqref{eq:H+meas} is subanalytic, so \eqref{eq:H+meas} must contain a subanalytic set $U$ that is open in $\RR^m$.  Let $f:X\times\RR^m\to\RR$ be the characteristic function of $\{x\}\times U$, which is subanalytic, so in particular is constructible.  Then for all $\lambda\in\RR$,
\[
|F(x,\lambda a)| = \left|\int_{\RR^m} f(x,y) e^{i\lambda a\cdot\phi(x,y)}dy\right|
= \left|\int_{U}e^{i\lambda b} dy\right| = \vol_m(U).
\]
So $|F(x,\lambda a)|$ does not tend to $0$ as $\lambda\to +\infty$, which shows that \eqref{eq:decayBndIFF} cannot hold.
\end{proof}

\begin{remark}\label{rmk:integFT}
Let $f:\RR\to\CC$ be integrable, and suppose that the real and imaginary components of $f$ are constructible.  Then the Fourier transform of $f$ is integrable if and only if $f$ is equivalent almost everywhere to a continuous function.
\end{remark}

\begin{proof}[Proof of Remark \ref{rmk:integFT}]
First suppose that $\hat{f}$ is integrable.  Then $\check{\hat{f}}$ is a continuous function, and the Fourier inversion theorem shows that $f(y) = \check{\hat{f}}(y)$ for almost all $y\in\RR$.

Conversely, suppose that there is a continuous function $g:\RR\to\CC$ such that $f(y) = g(y)$ for almost all $y\in\RR$.  Since the real and imaginary components of $f$ are constructible, $f = g$ at all but finitely many points, so the real and imaginary components of $g$ are also constructible.  Therefore applying Theorem \ref{thm:integFT} to the real and imaginary parts of $g$ shows that $\hat{g}$ is integrable, so $\hat{f}$ is as well because $\hat{f} = \hat{g}$.
\end{proof}